\documentclass[10pt]{amsart}
\usepackage{graphicx}
\usepackage{amsmath}
\usepackage{amssymb}
\usepackage{xcolor}
\usepackage{combelow} 
\usepackage{tensor} 
\usepackage{hyperref} 
\usepackage[all]{xy} 
\usepackage{indentfirst} 

\usepackage[a4paper]{geometry}
\newtheorem{theorem}{Theorem}[section]

\newtheorem{proposition}[theorem]{Proposition}

\theoremstyle{definition}
\newtheorem{remark}[theorem]{Remark}

\newtheorem{subsec}[theorem]{}

\numberwithin{equation}{section}

\begin{document}

\hfuzz=13pt
\title[An exact sequence for the graded Picent]{An exact sequence for the graded Picent}
\author[A. Marcus]{Andrei Marcus}
\author[V. A. Minu\cb{t}\u{a}]{Virgilius-Aurelian Minu\cb{t}\u{a}}
\address{\parbox{\linewidth}{Babe\cb{s}-Bolyai University,\\
Faculty of Mathematics and Computer Science,\\
Department of Mathematics,\\
1 Kog\u{a}lniceanu Street, 400084,\\
Cluj-Napoca, Romania\vspace{0.5cm}}}
\email{marcus@math.ubbcluj.ro}
\email{virgilius.minuta@ubbcluj.ro}



\begin{abstract}  To a strongly $G$-graded algebra $A$ with $1$-component $B$ we associate the group $\mathrm{Picent}^{\mathrm{gr}}(A)$ of isomorphism classes of invertible $G$-graded $(A,A)$-bimodules over the centralizer of $B$ in $A$. Our main result is a $\mathrm{Picent}$ version of the Beattie-del R\'{\i}o exact sequence, involving Dade's group $G[B]$, which relates $\mathrm{Picent}^{\mathrm{gr}}(A)$, $\mathrm{Picent}(B)$, and group cohomology.
\end{abstract}

\maketitle

\section{Introduction}

In this paper we introduce a notion of $G$-graded $\mathrm{Picent}$ group associated to a strongly $G$-graded algebra, and we establish an exact sequence linking this group and the $\mathrm{Picent}$ group of the $1$-component.

To explain our results, let us introduce some notation. Here $G$ is a finite group, $k$ a commutative ring, and we assume that the $k$-algebra $A$ is a $G$-graded crossed product with $1$-component $B$. There is an inertia group denoted by $G[B]$, introduced by E. C. Dade in \cite[(2.9)]{art:Dade1973} as the largest (normal) subgroup $H$ of $G$ such that the $G$-acted algebra $\mathcal{C}:=\mathrm{C}_A(B)_H$ is also a crossed product (see  \ref{ss:isoFR}). Let $Z=\mathrm{U}(\mathrm{Z}(B))$ and $H=G[B]$.

We define in Section 3 the group $\mathrm{Picent}^{\mathrm{gr}}(A)$ of the isomorphism classes of $G$-graded invertible $(A,A)$-bimodules over $\mathcal{C}$ (see \ref{ss:Pic_gr_and_Picent_gr_and_bimodules_over_C}). We show in Theorem \ref{th:main_exact_sequnce} that there is an exact sequence of groups
\[\xymatrix@C+=1cm{
1 \ar@{->}[r] & \mathrm{H}^1(G/H, Z) \ar@{->}[r]^{\Phi} & \mathrm{Picent}^{\mathrm{gr}}(A)\ar@{->}[r]^{\Psi} & \mathrm{Picent}(B)^{G/H}\ar@{->}[r]^{\Theta} & \mathrm{H}^2(G/H, Z).
}\]
This is a Picent version of the Beattie-del R\'{\i}o exact sequence from \cite{art:delRio1996}, where the Picard groups are over $k$. Note that in \cite{art:delRio1996}, the third map of the sequence is not a group homomorphisms, and the group $H$ does not play any role. Moreover, we show in Theorem \ref{th:diagram} that there is an interesting connection between the two exact sequences,  which involves the 5-term exact sequence associated to the normal subgroup $H$ of $G$.

Our proof of Theorem \ref{th:main_exact_sequnce} builds on the ideas of the proof given in \cite{art:Marcus1998} for the Beattie-del Rio exact sequence, which is based on the methods developed by Dade in \cite{art:Dade1981}. One of our motivations to revisit this result is the recent interest in the Picard groups of group algebras, see for instance, \cite{art:BKL2020} and \cite{art:EL2020}.

This paper is organized as follows. In Section 2 we recall some facts on crossed products and group cohomology. In Section 3 we introduce the group $\mathrm{Picent}^{\mathrm{gr}}(A)$ and prove several of its properties, which are generalizations of some well-known properties of the usual Picent group.
Our main results are in Section 4. We already mentioned two of them. Theorem \ref{th:butterfly} states that the exact sequence from Theorem \ref{th:main_exact_sequnce} only depends on the action on $B$ of the group of homogeneous units of $A$. This is closely related to Sp\"ath's ``butterfly theorem", see \cite[Theorem 2.16]{ch:Spath2018}, and also \cite[Proposition 4.1 and Theorem 4.2]{art:MM1}, respectively \cite[Proposition 3.8 and Theorem 3.9]{art:MM2}. In fact, trying to generalize these results was another motivation for the current paper.

Concerning our basic assumptions, all rings in this paper are associative with identity, and all modules are left (unless otherwise specified) and unital. We refer the reader to \cite{book:Reiner2003} for results on Picard groups, to \cite{book:Neukirch2008} for group cohomology, and to \cite{book:Marcus1999} for group graded algebras.

\section{Preliminaries}

\begin{subsec}\label{subsec:Groups} Let $G$  be a finite group, let $k$ be a commutative ring, and let $A$ be a crossed product of the $k$-algebra $B$ with $G$. We denote by $\mathrm{hU}(A)$ the group of homogeneous units of $A$. For all $g\in G$, we choose $u_g\in \mathrm{hU}(A)\cap A_g$.
\end{subsec}

\begin{subsec} \label{ss:isoFR} Consider the centralizer $\mathrm{C}_{A}(B)$ in $A$ of the $1$-component $B$. We know that $\mathrm{C}_{A}(B)$ is a $G$-graded $G$-algebra, and let
\[ G[B]:=\{  g\in   G\mid A_{  g}\simeq B \textrm{ as } (B,B)\textrm{-bimodules}\}\]
be the inertia group of $B$ as a $(B,B)$-bimodule. Then it is well known that $G[B]$ is a normal subgroup of $G$, and
\[\mathcal{C}:=\mathrm{C}_A(B)_{G[B]}=\bigoplus_{  h\in   G[B]}\mathrm{C}_A(B)_{ h}\]
is a strongly $G[B]$-graded $G$-acted subalgebra of $\mathrm{C}_A(B)$. Actually, $G[B]$ is the largest subgroup $H$ of $G$ such that $\mathrm{C}_A(B)_H$ is strongly graded. The action of $G$ on $\mathcal{C}$ is given by $$\tensor*[^{g}]{c}{}=u_g c u_{g}^{-1},$$ for any $c\in\mathcal{C}$ and $g\in G$. Note that this does not depend on the choice of $u_g\in \mathrm{hU}(A)\cap A_{g}$.

We denote $H:=G[B]$, $\mathcal{Z}:=\mathcal{C}_1=\mathrm{Z}(B)$ and $Z=\mathrm{U}(\mathcal{Z)}$.
\end{subsec}
\goodbreak

\begin{subsec} \label{ss:cvasihuge_diagram}
We recall alternative definitions of $H$.
We consider the group homomorphism $\varepsilon:\mathrm{hU}(A)\to \mathrm{Aut}_k(B)$, $\varepsilon(u_g)(b)=u_gbu_g^{-1}$, for all $b\in B$ and $g\in G$. Let $\bar{\varepsilon}:G\to\mathrm{Out}_k(B)$ be the group homomorphism induced by $\varepsilon$. Note that the definition of $\bar{\varepsilon}$ does not depend on the choice of $u_g\in \mathrm{hU}(A)\cap A_{g}$. We obtain the following commutative diagram of groups and group homomorphisms with exact rows and columns, where $\pi:\mathrm{Aut}_k(B)\to\mathrm{Out}_k(B)$ is the canonical projection:
\[
\xymatrix@C+=2cm{
1 \ar@{->}[r] & Z \ar@{->}[r] \ar@{->}[d] & \mathrm{C}_{\mathrm{hU}(A)}(B) \ar@{->}[r] \ar@{->}[d] & H \ar@{->}[r] \ar@{->}[d] & 1\\
1 \ar@{->}[r] & \mathrm{U}(B) \ar@{->}[r] \ar@{->}[d] & \mathrm{hU}(A) \ar@{->}[r] \ar@{->}[d]^{\varepsilon}  & G \ar@{->}[r] \ar@{->}[d]^{\bar{\varepsilon}} & 1\\
1 \ar@{->}[r] & \mathrm{Int}(B) \ar@{->}[r] & \mathrm{Aut}_k(B) \ar@{->}[r]^{\pi} & \mathrm{Out}_k(B) \ar@{->}[r] & 1.
}
\]

We have that
$$
\begin{array}{rclcl}
H & = & \mathrm{Ker}\,\bar{\varepsilon} & = & \{g\in G\mid \mathcal{C}_g\mathcal{C}_{g^{-1}}=\mathcal{Z}\},
\end{array}$$
and we obtain the group isomorphisms $\mathrm{Im}\,\bar{\varepsilon}\simeq G/H$ and
\[
    H  \simeq  \mathrm{C}_{\mathrm{hU}(A)}(B)/\mathrm{U}(\mathcal{Z})  \simeq \mathrm{C}_{\mathrm{hU}(A)}(B)\mathrm{U}(B)/\mathrm{U}(B).
\]
Furthermore, since $G\simeq \mathrm{hU}(A)/\mathrm{U}(B)$, we also obtain \[G/H\simeq \mathrm{hU}(A)/\mathrm{C}_{\mathrm{hU}(A)}(B)\mathrm{U}(B).\]

Note that $H$ acts trivially on $\mathcal{Z}$. Indeed, if $h\in H$ and $z\in \mathcal{Z}$, then $\tensor*[^{h}]{z}{}=u_hzu_{h}^{-1}$, where we may choose $u_h\in \mathrm{C}_{\mathrm{hU}(A)}(B)$, so the statement follows.
\end{subsec}

\begin{subsec} \label{ss:first_column}
In particular, $Z$ is a trivial $H$-module, hence a $G/H$-module. Recall (see for example \cite[Proposition 1.6.7]{book:Neukirch2008}) that we have the 5-term exact sequence:
\[\xymatrix@C+=1cm{
1 \ar@{->}[r] & \mathrm{H}^1(G/H, Z) \ar@{->}[r]^{\mathrm{Inf}} & \mathrm{H}^1(G, Z) \ar@{->}[r]^{\hspace{-5mm}\mathrm{Res}} & \mathrm{H}^1(H, Z)^{G/H} \ar@{->}[r]^{\mathrm{Tg}} & \mathrm{H}^2(G/H, Z) \ar@{->}[r]^{\mathrm{Inf}} & \mathrm{H}^2(G, Z).
}\]
In our situation, because $Z$ is a trivial $H$-module, we know from \cite[Corollary 6.4.6]{book:Weibel1994} that
$$\mathrm{H}^1(H, Z)=\mathrm{Hom}_{\mathrm{Grp}}(H,Z),$$
and the action of $G$ on $\mathrm{Hom}_{\mathrm{Grp}}(H,Z)$ is
\[{}^{g}\beta : H\to Z,\qquad ({}^{g}\beta)(h)={}^{g}(\beta({}^{g^{-1}}h)),\text{ for all }h\in H.\]
It is straightforward to see that $H$ acts trivially on $\mathrm{Hom}_{\mathrm{Grp}}(H,Z)$, hence $\mathrm{Hom}_{\mathrm{Grp}}(H,Z)$ is $G/H$-acted.
\end{subsec}

\section{Graded Picent} \label{sec:a}
\begin{subsec}
We denote by $\mathrm{Pic}_k(B)$ the set of isomorphism classes $[P]$ of invertible $(B,B)$-bimodules $P$ over $k$. Together with the operation
\[
[P]\cdot [Q]=[P\otimes_B Q],
\]
$\mathrm{Pic}_k(B)$ forms a group, called the Picard group of $B$ relative to $k$ (see for example \cite[Definition 37.5 (i)]{book:Reiner2003}). The inverse $[P]^{-1}$ of [P] is given by the $B$-dual $[\mathrm{Hom}_B(P,B)]$ of $P$.

Recall (see \cite[Definition 37.5 (ii)]{book:Reiner2003}) that $\mathrm{Picent}(B)$ is the set of isomorphism classes $[P]$ of invertible $(B,B)$-bimodules $P$ satisfying $zp=pz$ for all $p\in P$ and $z\in \mathcal{Z}$ (we will say that $P$ is a bimodule over $\mathcal{Z}$). Then $\mathrm{Picent}(B)$ is a subgroup of $\mathrm{Pic}_k(B)$.
\end{subsec}

\begin{subsec} \label{ss:Pic_gr_and_Picent_gr_and_bimodules_over_C}
We define the groups $\mathrm{Pic}^{\mathrm{gr}}_k(A)$ and  $\mathrm{Picent}^{\mathrm{gr}}(A)$ in a similar manner. Let $\mathrm{Pic}^{\mathrm{gr}}_k(A)$ be the group of isomorphism classes $[\tilde{P}]$ of $G$-graded invertible $(A,A)$-bimodules $\tilde{P}$ over $k$; and let $\mathrm{Picent}^{\mathrm{gr}}(A)$ be the subgroup of $\mathrm{Pic}^{\mathrm{gr}}_k(A)$, consisting of isomorphism classes $[\tilde{P}]$ of $G$-graded $(A,A)$-bimodules $\tilde{P}$ satisfying:
\begin{enumerate}
\item $\tilde{P}\otimes_A\tilde{P}^{\ast}\simeq A$ and $\tilde{P}^{\ast}\otimes_A \tilde{P}\simeq A$ as $G$-graded $(A,A)$-bimodules, where $\tilde{P}^{\ast}$ is the $A$-dual of $\tilde{P}$.
\item $\tilde{p}_g c=\tensor*[^g]{c}{}\tilde{p}_g$ for all $g\in G$, $c\in \mathcal{C}$, $\tilde{p}_g\in \tilde{P}_g$.
\end{enumerate}

Recall from \cite[Remark 2.6]{art:MM2} that condition (2) is equivalent to
\begin{enumerate}
\item[(2')] $p c=c p$ for all $c\in C$ and $p\in\tilde{P}_1$.
\end{enumerate}
Such bimodules are called in \cite[Definition 2.5]{art:MM2} $G$-graded $(A,A)$-bimodules over $\mathcal{C}$ and they form a category denoted by $A\textrm{-}\mathrm{Gr}/\mathcal{C}\textrm{-}A$.
\end{subsec}

\begin{subsec}
We recall some further notions and results from \cite[Section 2]{art:MM2}.
Denote $$\Delta^{\mathcal{C}}=\Delta(A\otimes_{\mathcal{C}} A^{\mathrm{op}})=\bigoplus_{g\in G}A_g\otimes_{\mathcal{C}}A_g^{\mathrm{op}}.$$
Then, by \cite[Propositions 2.10 and 2.11]{art:MM2}, the functors:
$$(A\otimes_{\mathcal{C}}A^{\mathrm{op}})\otimes_{\Delta^{\mathcal{C}}}-,\,A\otimes_B -,\,-\otimes_B A : \Delta^{\mathcal{C}}\textrm{-}\mathrm{mod} \to A\textrm{-}\mathrm{Gr}/\mathcal{C}\textrm{-}A,$$
are naturally isomorphic equivalences of categories, with inverse taking the 1-component, $(-)_1$, and are compatible with tensor products.
\end{subsec}

\begin{subsec}
We denote
\[\mathrm{Aut}^{\mathrm{gr}}_{\mathcal{C}}(A)=\{\alpha\in\mathrm{Aut}_k(A)\mid \alpha(A_g)=A_g \text{ for all } g\in G, \text{ and }\alpha(c)=c \text{ for all }  c\in\mathcal{C}\}.\]
Let $\tilde{P}$ be a $G$-graded $(A,A)$-bimodule over $\mathcal{C}$. For $\alpha,\beta\in\mathrm{Aut}^{\mathrm{gr}}_{\mathcal{C}}(A)$, we define a new $G$-graded $(A,A)$-bimodule over $\mathcal{C}$, where $_{\alpha}\tilde{P}_{\beta}=\tilde{P}$ as $k$-modules and with multiplication given by $$a\cdot \tilde{p}\cdot a'=\alpha(a)\tilde{p}\beta(a').$$

For $u\in\mathrm{U}(A)$, denote $\iota_u:A\to A$, $\iota_u(a)=uau^{-1}$. Denote
$$\mathrm{Int}_1(A)=\{\iota_u\mid u\in\mathrm{U}(B)\}.$$
\end{subsec}

The next result generalizes Theorems (37.14) and (37.16) from \cite{book:Reiner2003}.

\begin{proposition} With the above notations, we have:

\rm{1)} There is an exact sequence of groups
\[
\xymatrix{
1 \ar@{->}[r] & \mathrm{Int}_1(A) \ar@{->}[r] & \mathrm{Aut}^{\mathrm{gr}}_{\mathcal{C}}(A) \ar@{->}[r]^{\hspace{-0.3cm}\varphi} & \mathrm{Picent}^{\mathrm{gr}}(A),
}\]
where $\varphi(\alpha)=[\tensor*[_{1}]{A}{_{\alpha}}]$.

 \rm{2)} If $[\tilde{P}]$, $[\tilde{Q}]\in \mathrm{Picent}^{\mathrm{gr}}(A)$, then $\tilde{P}\simeq \tilde{Q}$ as a $G$-graded  $(A,\mathcal{C})$-bimodule over $\mathcal{C}$ if and only if there is $\alpha \in \mathrm{Aut}^{\mathrm{gr}}_{\mathcal{C}}(A)$ such that $\tilde{Q}\simeq {}_{1}\tilde{P}_{\alpha}$ as $G$-graded $(A,A)$-bimodules over $\mathcal{C}$. In particular, \[\mathrm{Im}\,\varphi=\{[\tilde{P}]\in \mathrm{Picent}^{\mathrm{gr}}(A)\mid\tensor*[]{\tilde{P}}{}\simeq \tensor*[]{A}{} \text{ as }G\text{-graded }(A,\mathcal{C})\text{-bimodules over }\mathcal{C}\}.\]
\end{proposition}
\begin{proof}
\rm{1)} It is clear that $\varphi$ is well-defined and a group homomorphism (see the proof of \cite[Theorem 37.14]{book:Reiner2003}). If $\alpha=\iota_u\in \mathrm{Int}_1(A)$, then the map $f:A\to {}_{1}A_{\alpha}$, $f(a)=au^{-1}$ is an isomorphism of $G$-graded $(A,A)$-bimodules, hence $\alpha\in \mathrm{Ker}\,\varphi$. Conversely, let $\alpha\in \mathrm{Aut}^{\mathrm{gr}}_{\mathcal{C}}(A)$ such that $f:A\to {}_{1}A_{\alpha}$ is an isomorphism of $G$-graded $(A,A)$-bimodules, that is
$$f(axa')=af(x)\alpha(a'),$$
for all $a,a',x\in A$. Let $u:=f(1)$, so $u\in B$, since $1\in B$. We deduce that $u\in \mathrm{U}(B)$ and $f(a)=au$, for all $a\in A$. Moreover, we get that $\alpha=\iota_{u^{-1}}\in\mathrm{Int}_1(A)$.

\rm{2)}  If $\tilde{Q}\simeq {}_{1}\tilde{P}_{\alpha}$ as $G$-graded $(A,A)$-bimodules over $\mathcal{C}$ for some $\alpha\in\mathrm{Aut}^{\mathrm{gr}}_{\mathcal{C}}(A)$, then obviously, $\tilde{Q}\simeq \tilde{P}$ as $G$-graded $(A,\mathcal{C})$-bimodules over $\mathcal{C}$.

Conversely, let $f:\tilde{P}\to \tilde{Q}$ be an isomorphism of $G$-graded $(A,\mathcal{C})$-bimodules. Consider the $k$-algebra isomorphism
\[f^{\ast}:\mathrm{Hom}_A({}_A \tilde{P}, {}_A \tilde{P})\to \mathrm{Hom}_A({}_A \tilde{Q}, {}_A \tilde{Q}),\qquad f^{\ast}(\rho)=f\circ\rho\circ f^{-1}.\]
Since $\tilde{P}$ is an invertible $(A,A)$-bimodule, any element of $\mathrm{Hom}_A(\tilde{P},\tilde{P})$ is uniquely of the form
$$\rho_{\tilde{P},a}:\tilde{P}\to \tilde{P},\qquad\rho_{\tilde{P},a}(x)=xa,$$
for all $x \in \tilde{P}$, where $a\in A$.
Similarly, any element of $\mathrm{Hom}_A(\tilde{Q},\tilde{Q})$ is uniquely of the form
$$\rho_{\tilde{Q},a}:\tilde{Q}\to \tilde{Q},\qquad \rho_{\tilde{Q},a}(y)=ya,$$
for all $y \in \tilde{Q}$, where $a\in A$.
Let $a\in A$ and let $a'\in A$ such that $f^{\ast}(\rho_{\tilde{P},a})=\rho_{\tilde{Q},a'}$.
We define
$$\alpha:A\to A,\qquad\alpha(a)=a',$$
so $\alpha$ is the $k$-algebra automorphism of $A$ corresponding to $f^{\ast}$.
This means that $(f\circ \rho_{\tilde{P},a}\circ f^{-1})(y)=\rho_{\tilde{Q},\alpha(a)}(y)$, for all $y\in \tilde{Q}$, that is
\[f(xa)=f(x)\alpha(a),\tag{$\ast$}\label{eq:unu}\]
for all $x\in \tilde{P}$.  In (\ref{eq:unu}), let $x\in \tilde{P}_h$ and $a\in A_g$, where $h,g\in G$. Then $f(x)\in \tilde{Q}_h$ and $f(xa)\in \tilde{Q}_{hg}$, so we must have that $\alpha(a)\in A_g$, hence $\alpha$ is grade-preserving.
By (\ref{eq:unu}) and by the assumption on $f$, for any $c\in\mathcal{C}$, we have
\[f(x)c=f(xc)=f(x)\alpha(c),\]
for all $x\in \tilde{P}$, hence $yc=y\alpha(c)$ for all $y\in \tilde{Q}$. It follows that $\alpha(c)=c$ for all $c\in\mathcal{C}$, and therefore $\alpha\in\mathrm{Aut}^{\mathrm{gr}}_{\mathcal{C}}(A)$. Consequently, the equality (\ref{eq:unu}) implies that $f:\tilde{P}\to \tilde{Q}_{\alpha}$ is an isomorphism of $G$-graded $(A,A)$-bimodules over $\mathcal{C}$.

\end{proof}

\begin{subsec} \label{subsec:P_tilda}
Assume that $P$ is a left $B$-module. Denote $\tilde{P}:=A\otimes_B P$, so $\tilde{P}$ is $G$-graded $(A,\mathrm{End}_A(\tilde{P})^{\mathrm{op}})$-bimodule. We know by \cite[Lemma 3.1]{art:MM1}, that, there is a $k$-algebra map from $\mathcal{C}\to\mathrm{End}_A(\tilde{P})^{\mathrm{op}}$. By restriction of scalars, we may regard $\hat P:=\tilde{P}_{H}=A_{H}\otimes_B P$ as an $H$-graded $(\mathcal{C},\mathcal{C})$-bimodule, with 1-component $P$.

Therefore, we may introduce the  group $\mathrm{Pic}_{\mathcal{C}}(B)$  of isomorphism classes of invertible $(B,B)$-bimodules $P$ satisfying $cp=pc$, for all $c\in\mathcal{C}$,
 $p\in P$.

Note that there are also other $H$-graded $(A_H,\mathcal{C})$-bimodule structures on $A_H\otimes_B P$, which will be used in the proof of  Theorem \ref{th:diagram}.
\end{subsec}

If we start with a $(B,B)$-bimodule over $\mathcal{Z}$ we obtain the following isomorphisms, which are related to \cite[Proposition 4.1]{art:MM1} and \cite[Proposition 3.8]{art:MM2}.

\begin{proposition} \label{prop:eq} There are $G/H$-equivariant isomorphisms
\[\mathrm{Picent}(B)\simeq \mathrm{Pic}_{\mathcal{C}}(B)\simeq \mathrm{Picent}^{\mathrm{gr}}(A_{H}).\]
\end{proposition}

\begin{proof}
Let $[P]\in \mathrm{Picent}(B)$, so $P$ is a $(B,B)$-bimodule over $\mathcal{Z}$. We regard $\hat{P}:=A_{H}\otimes_B P$ as an $H$-graded $(A_{H},\mathcal{C})$-bimodule, as in \ref{subsec:P_tilda}. Then there is a unique structure of an $H$-graded $(A_{H},A_{H})$-bimodule over $\mathcal{C}$ on $\hat{P}$, which is defined as follows: Let $g,h\in H$, $a_g\in A_g$, $a_h\in A_h$ and $p\in P$. Since $\mathcal{C}$ is a crossed product, we may write
$a_h=c_h b$, where $c_h\in \mathcal{C}_h\cap \mathrm{U}(\mathcal{C})$ and $b\in B$. We define
\[(a_g\otimes_B p)\cdot a_h = (a_g\otimes_B p)\cdot c_h b := a_gc_h\otimes_B pb.\]
This is well-defined. Indeed, if $a_h=c'_hb'$, where $c'_h\in \mathcal{C}_h\cap \mathrm{U}(\mathcal{C})$ and $b'\in B$, we have that there exists $z\in Z$ such that $c'_h=c_hz$. From $a_h=c'_hb'=c_hb$, we obtain $zb'=b$. Now, we have
\[
\begin{array}{rcl}
a_gc'_h\otimes_B pb' & = & a_gc_hz\otimes_B pb'\\
& = & a_gc_h\otimes_B zpb'\\
& = & a_gc_h\otimes_B pzb' \qquad \text{ (since }P\text{ is over }\mathcal{Z})\\
& = & a_gc_h\otimes_B pb.
\end{array}
\]
In fact, $\hat{P}$ is a $H$-graded $(A_{H},A_{H})$-bimodule over $\mathcal{C}$. Indeed, let $g,h\in H$, $a_g\in A_g$, $c_h\in \mathcal{C}_h$ and $p\in P$. Because $\mathcal{C}$ is a crossed product, we may write
$c_h=u_h b$, where $u_h\in \mathcal{C}_h\cap \mathrm{U}(\mathcal{C})$ and $b\in B$. Note that, since $c_h\in \mathcal{C}_h$ and $u_h\in \mathcal{C}_h\cap \mathrm{U}(\mathcal{C})$, it is clear that $b\in \mathcal{Z}$. We have
\[
\begin{array}{rcl}
(a_g\otimes_B p)c_h & = & (a_g\otimes_B p) u_h b \\
& = &  a_gu_h\otimes_B pb\\
& = & a_gu_h\otimes_B bp \qquad \text{ (since }P\text{ is over }\mathcal{Z})\\
& = & a_gu_hb\otimes_B p\\
& = & a_gc_h\otimes_B p\\
& = & {}^{g}c_ha_g\otimes_B p \qquad \text{ (since }A\text{ is over }\mathcal{C})\\
& = & {}^{g}c_h(a_g\otimes_B p).
\end{array}
\]

Since $P$ induces a Morita autoequivalence over $\mathcal{Z}$ of $B$, by \cite[Theorem 3.3]{art:MM2}, we have that $\hat{P}$ induces an $H$-graded Morita autoequivalence over $\mathcal{C}$ of $A_{H}$.  Observe that, by the construction of $\hat{P}$, we obtain $pc=cp$ for all $p\in P$ and $c\in \mathcal{C}$. This gives the isomorphism $\mathrm{Picent}(B)\simeq \mathrm{Pic}_{\mathcal{C}}(B)$.

If $f:P\to Q$ is an isomorphism of $(B,B)$-bimodules, then $\hat{f}:=A_H\otimes_B f: A_H\otimes_B P \to A_H\otimes_B Q$ is an isomorphism of $H$-graded $(A_H,B)$-bimodules, and it is easy to see that it is actually an isomorphism of $H$-graded $(A_H,A_H)$-bimodules.  We deduce that the map
\[
\mathrm{Pic}_{\mathcal{C}}(B)\to \mathrm{Picent}^{\mathrm{gr}}(A_{H}),\qquad [P] \mapsto [\hat{P}]
\]
is well-defined and it is a bijection with inverse given by $[\hat{P}]\mapsto [\hat{P}_1]$.
By \cite[Proposition 2.11 (1)]{art:MM2} we deduce that this map is a group isomorphism.

The group $G$ acts on $\mathrm{Pic}_{k}(B)$, on $\mathrm{Picent}(B)$, and on $\mathrm{Pic}_{\mathcal{C}}(B)$ by defining $\tensor*[^{g}]{[P]}{}=[A_g\otimes_B P\otimes_B A_{g^{-1}}]$. The above argument shows that $H$ acts trivially on $\mathrm{Pic}_{\mathcal{C}}(B)$, hence $G/H$ acts on $\mathrm{Pic}_{\mathcal{C}}(B)$.
Similarly, $G/H$ acts on $\mathrm{Picent}^{\mathrm{gr}}(A_H)$ by defining
\[{}^{gH}[\hat{P}]:=[A_{gH}\otimes_{A_H}\hat{P}\otimes_{A_H}A_{g^{-1}H}].\]
But, $A_{gH}\otimes_{A_H}\hat{P}\otimes_{A_H}A_{g^{-1}H}\simeq A_H\otimes_B (A_g\otimes_B P \otimes_B A_{g^{-1}})$ as $H$-graded $(A_H,A_H)$-bimodules, hence the isomorphisms from the statement are $G/H$-equivariant.
\end{proof}

\begin{subsec} We denote by $\mathrm{Aut}_{kG}^{\mathrm{gr}}(\mathcal{C})$ the set of all $k$-linear automorphisms of $\mathcal{C}$, which preserve gradings and are also $G$-algebra maps.
\end{subsec}

\begin{proposition} \label{prop:gr_pic_action}
There is an exact sequence of groups:
\[
\xymatrix@C+=1cm{
1 \ar@{->}[r] & \mathrm{Picent}^{\mathrm{gr}}(A) \ar@{->}[r] & \mathrm{Pic}^{\mathrm{gr}}_{k}(A) \ar@{->}[r]^{\tilde{\sigma}} & \mathrm{Aut}_{kG}^{\mathrm{gr}}(\mathcal{C}),
}
\]
where $\tilde{\sigma}([\tilde{P}])=\tilde{\sigma}_{\tilde{P}}:\mathcal{C}\to\mathcal{C}$ is an automorphism of $H$-graded $G$-acted $k$-algebras, determined by $\tilde{\sigma}_{\tilde{P}}(c)p=pc$ for all $p\in P:=\tilde{P}_1$ and $c\in\mathcal{C}$.
\end{proposition}

\begin{proof}
Let $[\tilde{P}]\in\mathrm{Pic}^{\mathrm{gr}}_{k}(A)$.
We define a map $\tilde{\sigma}_{\tilde{P}}:\mathcal{C}\to\mathcal{C}$ as follows.
We know that $\mathcal{C}=\mathcal{C}_H$, so let $h\in H$ and $c\in\mathcal{C}_h$. Recall that the $h$-suspension of a $G$-graded $A$-module $\tilde{M}$ is defined to be the $G$-graded $A$-module $\tilde{M}(h)=\bigoplus_{g\in G}\tilde{M}(h)_g$, where $\tilde{M}(h)_g=\tilde{M}_{gh}$.
The left multiplication $\lambda_c: P\to \tilde{P}_h$ is $B$-linear, hence
\[A\otimes_B \lambda_c: \tilde{P}\simeq A\otimes_B P\to \tilde{P}(h)\simeq A\otimes_B \tilde{P}_h \]
is a grade-preserving $A$-linear map. Since $\tilde{P}$ is an invertible $(A,A)$-bimodule, there is a unique $a\in A$ such that $A\otimes_B\lambda_c=\tilde{\rho}_{a}$, where $\tilde{\rho}_{a}:\tilde{P}\to\tilde{P}$ denotes the right multiplication by $a$. Since $A\otimes_B \lambda_c:\tilde{P}\to\tilde{P}(h)$, is grade-preserving, it follows that $a\in A_h$.
Note that the equality $A\otimes_B\lambda_c=\tilde{\rho}_{a}$ is equivalent to  $cp=pa$, for all $p\in P$. Now, for all $b\in B$,
\[
\begin{array}{rclclclclclclclclcl}
p(ab) & = & (pa)b &=& (cp)b &=& c(pb)&=& (pb)a &=& p(ba).
\end{array}
\]
This implies that $\tilde{\rho}_{ab}=\tilde{\rho}_{ba}$, hence $ab=ba$, for all $b\in B$, that is $a\in \mathcal{C}_h$.

By a similar argument, the element $c\in\mathcal{C}$ determines an unique element $a\in A$ such that $\rho_{c}\otimes_B A=\tilde{\lambda}_{a}$, where $\tilde{\lambda}_{a}:\tilde{P}\to\tilde{P}$ denotes the left multiplication by $a$. Thus, we obtain a grade-preserving bijection $\tilde{\sigma}_{\tilde{P}}:\mathcal{C}\to\mathcal{C}$, and we choose the notation such that $\tilde{\sigma}_{\tilde{P}}(c)p=pc$ for all $p\in P$ and $c\in\mathcal{C}$.

Let $\tilde{f}:\tilde{P}\to \tilde{P}'$ be an isomorphism of $G$-graded $(A,A)$-bimodules and $c\in \mathcal{C}$. Let $a\in A$ be the unique element such that $ap=pc$, for all $p\in P$ and let $a'\in A$ be the unique element such that $a'p'=p'c$, for all $p'\in P':=\tilde{P}'_1$. We get $\tilde{f}(ap)=\tilde{f}(pc)$, hence $a\tilde{f}(p)=\tilde{f}(p)c$, for all $p\in P$. Since $\tilde{f}$ is grade preserving and surjective, we obtain $ap'=p'c$, for all $p'\in P'$. But $a'\in A'$ was unique with this property, hence $a=a'$. This implies that $\tilde{\sigma}_{\tilde{P}}$ depends only on the class $[\tilde{P}]$. We also have that $\tilde{\sigma}_{A}$ is the identity map of $\mathcal{C}$.

By construction, it is quite clear that $\tilde{\sigma}_{\tilde{P}}:\mathcal{C}\to\mathcal{C}$ is a $k$-algebra map. Moreover, with $p,c$ and $a$ as above, for all $g\in G$, we have
\[
\begin{array}{rcl}
u_g cu_g^{-1} p &=& u_g cu_g^{-1} p u_g u_g^{-1}\\
&=& u_g u_g^{-1} p u_g a u_g^{-1}\\
&=& p u_g a u_g^{-1}.
\end{array}
\]
This shows that $\tilde{\sigma}_{\tilde{P}}({}^{g}c)={}^g\tilde{\sigma}_{\tilde{P}}(c)$, so $\tilde{\sigma}_{\tilde{P}}$ is a $G$-algebra map.

To show that $\tilde{\sigma}:\mathrm{Pic}^{\mathrm{gr}}_{k}(A) \to \mathrm{Aut}^{\mathrm{gr}}_{kG}(\mathcal{C})$ is a group homomorphism, let $[\tilde{P}]$, $[\tilde{Q}]\in \mathrm{Pic}^{\mathrm{gr}}_{k}(A)$. We know that the map
\[
    P\otimes_B Q\to (\tilde{P}\otimes_A\tilde{Q})_1,\qquad p\otimes_B q\mapsto p\otimes_A q
\]
is an isomorphism of $\Delta(A\otimes_{k}A^{\mathrm{op}})$-modules (see \cite[Lemma 1.6.6 a)]{book:Marcus1999}).
Then, for any $p\in P:=\tilde{P}_1$, $q\in Q:=\tilde{Q}_1$ and $c\in\mathcal{C}$, we have:
\[
\begin{array}{rcl}
({p}\otimes_A{q})c & = & {p}\otimes_A{q}c\\
& = & {p}\otimes_A\tilde{\sigma}_{\tilde{Q}}(c){q}\\
& = & {p}\tilde{\sigma}_{\tilde{Q}}(c)\otimes_A{q}\\
& = & \tilde{\sigma}_{\tilde{P}}(\tilde{\sigma}_{\tilde{Q}}(c)){p}\otimes_A{q}\\
& = & (\tilde{\sigma}_{\tilde{P}}\circ \tilde{\sigma}_{\tilde{Q}})(c)({p}\otimes_A{q}),
\end{array}
\]
hence $\tilde{\sigma}_{\tilde{P}\otimes_A\tilde{Q}}=\tilde{\sigma}_{\tilde{P}}\circ\tilde{\sigma}_{\tilde{Q}}$.

Finally, we see that $\tilde{P}\in\mathrm{Ker}\,\tilde{\sigma}$ if and only if $\tilde{\sigma}_{\tilde{P}}=\mathrm{id}_{\mathcal{C}}$, that is, $cp=pc$, for all $c\in\mathcal{C}$ and $p\in P$,
which means that $\tilde{P}\in\mathrm{Picent}^{\mathrm{gr}}(A)$.
\end{proof}

\begin{remark} \label{ss:G_action}
In particular, taking $G=1$ in Proposition \ref{prop:gr_pic_action}, we obtain that $\mathrm{Pic}_{k}(B)$ acts on $\mathcal{Z}$, by $\sigma_{P}(z)p=pz$, for all $p\in P$ and $z\in \mathcal{Z}$. Because the map $G\to\mathrm{Pic}_{k}(B)$, $g\mapsto [A_g]$ is a group homomorphism, we obtain that $G$ acts on $\mathcal{Z}$, such that ${}^{g}z$ is defined by $a_g z={}^{g}z a_g$, for all $z\in\mathcal{Z}$, $g\in G$ and $a_g\in A_g$.

Even more, we may deduce from Proposition \ref{prop:gr_pic_action}, the action of $G$ on $\mathcal{C}$ given in \ref{ss:isoFR} as follows.
Since $A_g\simeq A_h$ as left $B$-modules for all $g,h\in G$, we get that $A_g$ is a $G$-invariant progenerator in $B\textrm{-}\mathrm{mod}$, hence $A\otimes_B A_g\simeq A(g)$ is a $G$-graded progenerator in $A\textrm{-}\mathrm{mod}$. Therefore, the map
\[G\to \mathrm{Pic}^{\mathrm{gr}}_{k}(A),\qquad g\mapsto [A(g)]\]
is a group homomorphism. It follows by Proposition \ref{prop:gr_pic_action}, that $G$ acts on $\mathcal{C}$ such that ${}^{g}c$ is defined by $a_g c={}^{g}c a_g$, for all $c\in\mathcal{C}$, $g\in G$ and $a_g\in A_g$.
\end{remark}

\begin{remark} If $[P]\in \mathrm{Picent}(B)^{G/H}$, that is $A_g\otimes_B P\otimes_B A_{g^{-1}}\simeq B$ as $(B,B)$-bimodules, then we may define yet another action of $G$ on $\mathcal{Z}$, but which actually does not depend on $P$, and coincides with the actions mentioned in \ref{ss:G_action}.

This action is defined as follows. Regard $P$ as a $\Delta_1$-module and consider the $G$-graded endomorphism algebra $E:=\mathrm{End}_{\Delta}(\Delta\otimes_{\Delta_1} P)^{\mathrm{op}}.$ The $G$-invariance of $P$ implies that $E$ is a crossed product of $E_1\simeq \mathcal{Z}$ and $G$, where $z\in \mathcal{Z}$ is identified with the right multiplication $\rho_z:P\to P$.
For any $g\in G$, choose an isomorphism \[f_g:P\to A_g\otimes_B P\otimes_B A_{g^{-1}}\] of $(B,B)$-bimodules, so $f_g$ can be regarded as an element of $\mathrm{U}(E)\cap E_g$. This gives the action of $G$ on $E_1$ by $(g,\rho_z)\mapsto f_g\circ \rho_z\circ f_{g^{-1}}$. It is not difficult to show that $f_g\circ \rho_z\circ f_{g^{-1}}=\rho_{\tensor*[^{g}]{z}{}}:P\to P$, hence the actions coincide.
\end{remark}

\section{The exact sequence}

We can now state the main results of this paper. Our first theorem is the Picent version of the Beattie-del R\'\i o exact sequence \cite{art:delRio1996}. We adapt here the proof from \cite{art:Marcus1998}.

\begin{theorem} \label{th:main_exact_sequnce}
There is an exact sequence of groups:
\[\xymatrix@C+=1cm{
1 \ar@{->}[r] & \mathrm{H}^1(G/H, Z) \ar@{->}[r]^{\Phi} & \mathrm{Picent}^{\mathrm{gr}}(A)\ar@{->}[r]^{\Psi} & \mathrm{Picent}(B)^{G/H}\ar@{->}[r]^{\Theta} & \mathrm{H}^2(G/H, Z).
}\]
\end{theorem}

\begin{proof} \textit{Step 1. The group homomorphism $\Psi:\mathrm{Picent}^{\mathrm{gr}}(A)\to \mathrm{Picent}(B)^{G/H}$.} Let $[\tilde{P}]\in\mathrm{Picent}^{\mathrm{gr}}(A)$, then $\tilde{P}$ is an $(A,A)$-bimodules over $\mathcal{C}$. By \cite[Theorem 3.3]{art:MM2}, the 1-component $P$ of $\tilde{P}$ is a $G$-invariant $(B,B)$-bimodule over $\mathcal{Z}=\mathcal{C}_1$, hence $[P]\in \mathrm{Picent}(B)^{G/H}$, so we define $\Psi([\tilde{P}])=[P]$. Clearly this is a well-defined map. If $[\tilde{P}],[\tilde{Q}]\in\mathrm{Picent}^{\mathrm{gr}}(A)$, then by \cite[Proposition 2.11]{art:MM2}, we have $(\tilde{P}\otimes_A \tilde{Q})_1\simeq P\otimes_B Q$, hence $\Psi$ is a group homomorphism.

\textit{Step 2. The group homomorphism  $\Phi: \mathrm{H}^1(G/H, Z) \to \mathrm{Picent}^{\mathrm{gr}}(A)$.} Consider the canonical $\Delta$-module $P=B$, where $\Delta=\Delta(A\otimes_k A^{\mathrm{op}})$ and let $E:=\mathrm{End}_{\Delta}(\Delta\otimes_{\Delta_1} P)^{\mathrm{op}}$. By a theorem of Dade (see \cite[Section 2]{art:Dade1981} or \cite[Theorem 3.1.1]{book:Marcus1999}), there is a bijection between the isomorphism classes of the $\Delta$-module structures on $B$ and the $Z$-conjugacy classes of splittings of $\mathrm{hU}(E)$. The  splitting $\mu:G\to \mathrm{hU}(E)$ corresponding to $P$ is induced by the multiplication maps $\mu_g:A_{g^{-1}}\otimes_B P\otimes_B A_g \to P$, $g\in G$. Moreover, $P$ is the 1-component of the $G$-graded $(A,A)$-bimodule $\tilde{P}:=A$, so $[\tilde{P}]$ is the identity element of $\mathrm{Picent}^{\mathrm{gr}}(A)$.

It is well-known that there is a bijection between $\mathrm{H}^1(G,Z)$ and the $Z$-conjugacy classes of splittings of $\mathrm{hU}(E)$, given by $[\gamma]\mapsto \bar{\gamma}:G\to \mathrm{hU}(E)$, where $\bar{\gamma}(g)=\gamma(g)\mu(g)$, for all $g\in G$ and all $\gamma\in\mathrm{Z}^{1}(G,Z)$.

We also obtain an injection between $\mathrm{H}^{1}(G,Z)$ and $\mathrm{Picent}^{\mathrm{gr}}(A)$ as follows. Denote by $P^{\gamma}$ the $\Delta$-module extending $P$, corresponding to $[\gamma]$. Then $\tilde{P}^{\gamma}=(A\otimes_k A^{\mathrm{op}})\otimes_{\Delta} P^{\gamma}$ represents a class in $\mathrm{Pic}^{\mathrm{gr}}_k(A)$. It is clear that $\tilde{P}^{\gamma}$ is a $G$-graded $(A,A)$-bimodule over $\mathcal{C}$ if and only if $\tilde{P}^{\gamma}_H=A_H$. This happens if and only if $\tilde{\gamma}(h)=\mu (h)$, for all $h\in H$, or equivalently, the restriction of $\gamma$ to $H$ is trivial. By \ref{ss:first_column}, we may regard $[\gamma]\in\mathrm{H}^1(G/H,Z)$, and in this case $[\tilde{P}^{\gamma}]\in\mathrm{Picent}^{\mathrm{gr}}(A)$.

Since ${P}^{\gamma}\otimes_B {P}^{\delta}\simeq {P}^{\gamma\delta}$ as $\Delta$-modules, for all $\gamma,\delta \in \mathrm{Z}^1(G/H,Z)$, we conclude that $\Phi$: $[\gamma]\mapsto[\tilde{P}^{\gamma}]$ is a well-defined injective group homomorphism.

\textit{Step 3. Exactness at $\mathrm{Picent}^{\mathrm{gr}}(A)$.} It is clear by the definitions of $\Phi$ and $\Psi$ that a class $[\tilde{P}]$ belongs to the image of $\Phi$ if and only if the 1-component $P$ of $\tilde{P}$ is isomorphic to $B$ as $(B,B)$-bimodules, which holds if and only if $[\tilde{P}]$ belongs to $\mathrm{Ker}\,\Psi$.

\textit{Step 4. The group homomorphism $\Theta:\mathrm{Picent}(B)^{G/H}\to \mathrm{H}^2(G/H, Z)$.} Let $[P]\in\mathrm{Picent}(B)^{G/H}$, so $P$ is a $G$-invariant $(B,B)$-bimodule over $\mathcal{Z}$. Let $\Delta=\Delta(A\otimes_k A^{\mathrm{op}})$. By Proposition \ref{prop:eq}, $P$ extends to a $\Delta_H$-module, and this $\Delta_H$-module is also $G/H$-invariant.  We regard $\Delta$ as a $G/H$-graded algebra with 1-component $\Delta_H$.

Consider the $G/H$-graded endomorphism algebra $\bar{E}=\mathrm{End}_{\Delta}(\Delta\otimes_{\Delta_H}P)^{\mathrm{op}}$ with 1-component $\bar{E}_1\simeq \mathcal{Z}$. By the $G/H$-invariance of $P$, we have that $\bar{E}$ is actually a crossed product of $\mathcal{Z}$ and $G/H$. Let $\Theta([P])$ be the element of $\mathrm{H}^2(G/H, Z)$ associated to the extension:
\[
\xymatrix@C+=1cm{
1 \ar@{->}[r] & Z \ar@{->}[r] & \mathrm{hU}(E) \ar@{->}[r] & G/H \ar@{->}[r] & 1.
}
\]
It is clear that $\Theta$ is a well-defined map. To show that $\Theta$ is a group homomorphism, let $[P],[Q]\in\mathrm{Picent}(B)^{G/H}$. We regard $P$ and $Q$ as $G/H$-invariant $\Delta_H$-modules, and let
\[\alpha_{gH}:P\to\Delta_{gH}\otimes_{\Delta_H}P,\qquad \beta_{gH}:Q\to \Delta_{gH}\otimes_{\Delta_H}Q\]
be isomorphisms of $\Delta_H$-modules, for all $gH\in G/H$. Observe that $P\otimes_B Q$ is also a $\Delta_H$-module, and
\[\Delta_{gH}\otimes_{\Delta_H}(P\otimes_B Q)\simeq (\Delta_{gH}\otimes_{\Delta_H}P)\otimes_B(\Delta_{gH}\otimes_{\Delta_H}Q),\]
as $\Delta_H$-modules. Then we have the $\Delta_H$-isomorphisms
\[\gamma_{gH}:=\alpha_{gH}\otimes_B\beta_{gH}:P\otimes_B Q\to \Delta_{gH}\otimes_{\Delta_H}(P\otimes_B Q).\]
We have $\Theta([P])=[\alpha]$ and $\Theta([Q])=[\beta]$, where
\[\alpha({gH},g'H)=\alpha_{gH}\alpha_{g'H}\alpha_{gg'H}^{-1}\qquad\text{and}\qquad \beta({gH},g'H)=\beta_{gH}\beta_{g'H}\beta_{gg'H}^{-1}.\]
We obtain
\[
\begin{array}{rcl}
    (p\otimes q)\gamma(gH,g'H) & = & (p\otimes q)\gamma_{gH}\gamma_{g'H}\gamma_{gg'H}^{-1}\\
    & = & (p\otimes q)(\alpha_{gH}\otimes\beta_{gH})(\alpha_{g'H}\otimes\beta_{g'H})(\alpha_{gg'H}^{-1}\otimes\beta_{gg'H}^{-1})\\
    & = & (p)\alpha_{gH}\alpha_{g'H}\alpha_{gg'H}^{-1}\otimes (q)\beta_{gH}\beta_{g'H}\beta_{gg'H}^{-1}\\
    & = & p\cdot \alpha(gH,g'H) \otimes q\cdot \beta(gH,g'H)\\
    & = & p\otimes q\cdot \alpha(gH,g'H) \beta(gH,g'H),
\end{array}
\]
for all $gH,g'H\in G/H$. It follows that $\Theta([P]\cdot [Q])=\Theta([P])\cdot\Theta([Q])$.

  \textit{Step 5. Exactness at $\mathrm{Picent}(B)^{G/H}$.} Let $[P]\in \mathrm{Picent}(B)^{G/H}$. By Dade's theorem \cite[Theorem 2.8]{art:Dade1981}, $\Theta([P])$ is trivial if and only if $P$ extends to a $\Delta$-module, which is over $\mathcal{C}$ in our situation. By \cite[Theorem 3.3]{art:MM2}, this is equivalent to the existence of a class $[\tilde{P}]\in \mathrm{Picent}^{\mathrm{gr}}(A)$, such that $[P]=\Psi([\tilde{P}])$.
\end{proof}

\begin{theorem} \label{th:diagram} {\rm1)}     The squares in the following diagram, with exact rows and columns, are commutative:
    \[\xymatrix@C+=1cm{
& 1 \ar@{->}[d] & 1 \ar@{->}[d] & 1 \ar@{->}[d] & &\\
1 \ar@{->}[r] & \mathrm{H}^1(G/H, Z) \ar@{->}[r]^{\Phi} \ar@{->}[d]^{\mathrm{Inf}}& \mathrm{Picent}^{\mathrm{gr}}(A)\ar@{->}[r]^{\Psi} \ar@{->}[d]^{\subseteq} & \mathrm{Picent}(B)^{G/H}\ar@{->}[r]^{\Theta} \ar@{->}[d]^{\subseteq} & \mathrm{H}^2(G/H, Z) \ar@{->}[d]^{\mathrm{Inf}}&\\
1 \ar@{->}[r] & \mathrm{H}^1(G, Z) \ar@{->}[r]^{\Phi_k} \ar@{->}[d]^{\mathrm{Res}} & \mathrm{Pic}^{\mathrm{gr}}_k(A)\ar@{->}[r]^{\Psi_k} \ar@{->}[d]^{\tilde{\sigma}}& \mathrm{Pic}_k(B)^{G}\ar@{->}[r]^{\Theta_k} \ar@{->}[d]^{\sigma} & \mathrm{H}^2(G, Z)&\\
1 \ar@{->}[r] & \mathrm{H}^1(H, Z)^{G/H} \ar@{->}[r]^{\phi}
\ar`d/8pt[r]`[rrrr]_{\mathrm{Tg}}`[uuu]`[rrr][rrruu]
& \mathrm{Aut}^{\mathrm{gr}}_{kG}(\mathcal{C}) \ar@{->}[r]^{\psi} & \mathrm{Aut}_{k}(\mathcal{Z})^{G}&&
}\]

{\rm2)} There is a group homomorphism $\theta:\mathrm{H}^1(H,Z)^{G/H}\to \mathrm{Pic}^{\mathrm{gr}}_{\mathcal{Z}}(A_H)^{G/H}$ such that the following diagram is commutative.
\[
\xymatrix@C+=4cm{
& \mathrm{H}^1(H,Z)^{G/H} \ar@{->}[d]^{\mathrm{Tg}}\ar@{..>}[dl]_{\theta} \\
\mathrm{Pic}^{\mathrm{gr}}_{\mathcal{Z}}(A_H)^{G/H} \ar@{->}[r]^{\Theta} & \mathrm{H}^2(G/H,Z).
}\]
\end{theorem}

\goodbreak
\begin{proof} 1)  The first row of the diagram holds by Theorem \ref{th:main_exact_sequnce}, and the second row is a result obtained by Beattie and del R\'{i}o in \cite{art:delRio1996} (see \cite{art:Marcus1998} for a direct proof). We consider the third row.

\textit{Step 1: The map $\phi:\mathrm{H}^1(H, Z)^{G/H}\to \mathrm{Aut}^{\mathrm{gr}}_{kG}(\mathcal{C})$.} We know from \ref{ss:first_column} that $\mathrm{H}^1(H, Z)=\mathrm{Hom}_{\mathrm{Grp}}(H,Z)$. Consider $\beta\in \mathrm{Hom}_{\mathrm{Grp}}(H,Z)^{G/H}$. We define
\[f_{\beta}:\mathcal{C}\to\mathcal{C},\qquad f_{\beta}(c_h)=\beta(h) c_h,\]
for all $h\in H$, $c_h\in \mathcal{C}_h$. This map is well-defined. Indeed, for all $b\in B$ we have:
\[  \beta(h) c_h b  =  \beta(h) b c_h=  b \beta(h) c_h,\]
therefore we have that $\beta(h) c_h\in\mathrm{C}_A(B)$. It is clear that $\beta(h) c_h\in\mathrm{C}_A(B)_h$, hence, also because $h\in H$, we obtain $\beta(h) c_h\in\mathcal{C}_h$. Note that $f_{\beta}$ is $k$-linear, grade preserving, and the inverse of $f_{\beta}$ is the map $c_h\mapsto (\beta(h))^{-1} c_h$. Moreover, for all $h_1,h_2\in H$, $c_{h_1}\in \mathcal{C}_{h_1}$, $c_{h_2}\in \mathcal{C}_{h_2}$, we have:
\[
\begin{array}{rcl}
     f_{\beta}(c_{h_1} c_{h_2}) & = & \beta(h_1h_2) c_{h_1} c_{h_2}\\
     & = & \beta(h_1)\beta(h_2) c_{h_1} c_{h_2}\\
     & = & \beta(h_1) c_{h_1} \beta(h_2) c_{h_2},  \\ 
      & = & f_\beta(c_{h_1}) f_\beta(c_{h_2}),
\end{array}
\]
therefore $f_{\beta}\in \mathrm{Aut}^{\mathrm{gr}}_{k}(\mathcal{C})$.

Because $\beta\in \mathrm{Hom}_{\mathrm{Grp}}(H,Z)^{G/H}$, we have that $\beta$ is $G/H$-invariant, that is,  ${}^g\beta=\beta$, for all $g\in G$, which implies $({}^g\beta)(h)=\beta(h)$, for all $g\in G$, $h\in H$, that is, ${}^{g}(\beta({}^{g^{-1}}h))=\beta(h)$, for all $g\in G$, $h\in H$. We have:
\[\begin{array}{rcl}
     {}^{g}(f_{\beta}(c_h)) & = & u_g f_{\beta}(c_h) u_g^{-1}\\
     & = & u_g \beta(h) c_h u_g^{-1}\\
     & = & u_g \beta(g^{-1}ghg^{-1}g) u_g^{-1} u_g c_h u_g^{-1}\\
     & = & {}^{g}(\beta({}^{g^{-1}}(ghg^{-1}))) u_g^{-1} u_g c_h u_g^{-1}\\
     & = & \beta(ghg^{-1}) u_g c_h u_g^{-1}\\
     & = & f_\beta(u_g c_h u_g^{-1})\\
     & = & f_\beta({}^g c_h),
\end{array}\]
for all $g\in G$, $h\in H$, $c_h\in \mathcal{C}_h$. This concludes that $f_{\beta}\in \mathrm{Aut}^{\mathrm{gr}}_{kG}(\mathcal{C})$.

We define $\phi:\mathrm{H}^1(H, Z)^{G/H}\to \mathrm{Aut}^{\mathrm{gr}}_{kG}(\mathcal{C})$ by $\phi(\beta)=f_\beta$, as above. We prove that $\phi$ is a group homomorphism. Indeed, if $\beta_1,\beta_2\in \mathrm{Hom}_{\mathrm{Grp}}(H,Z)^{G/H}$, $h\in H$ and $c_h\in\mathcal{C}_h$, we have:
\[\begin{array}{rcl}
     \phi(\beta_1 \beta_2)(c_h) & = & f_{\beta_1 \beta_2}(c_h)\\
     & = & (\beta_1 \beta_2)(h) c_h\\
     & = & \beta_1(h) \beta_2(h) c_h\\
     & = & \beta_1(h)f_{\beta_2}(c_h)\\
     & = & f_{\beta_1}( f_{\beta_2}(c_h))\\
     & = & (f_{\beta_1}\circ f_{\beta_2})(c_h)\\
     & = & (\phi({\beta_1})\circ \phi({\beta_2}))(c_h).
\end{array}\]
It is also easy to see  that $\phi$ is injective.
\goodbreak
\textit{Step 2: The map $\psi:\mathrm{Aut}^{\mathrm{gr}}_{kG}(\mathcal{C})\to  \mathrm{Aut}_{k}(\mathcal{Z})^{G}$.} It is straightforward that for $f\in \mathrm{Aut}^{\mathrm{gr}}_{kG}(\mathcal{C})$, $f_1\in \mathrm{Aut}_{k}(\mathcal{Z})$. We prove that $f_1\in \mathrm{Aut}_{k}(\mathcal{Z})^G$. First, recall that because $G$ acts on $\mathcal{Z}$, we have that $G$ acts on $\mathrm{Aut}_{k}(\mathcal{Z})$ via $({}^g\alpha)(z)={}^{g}(\alpha({}^{g^{-1}}z))$, for all $z\in \mathcal{Z}$, $g\in G$, $\alpha\in \mathrm{Aut}_{k}(\mathcal{Z})$. Second, because $f\in \mathrm{Aut}^{\mathrm{gr}}_{kG}(\mathcal{C})$, we have that $f({}^{g}c)={}^{g}(f(c))$, for all $g\in G$, $c\in\mathcal{C}$. This implies that $f({}^{g}z)={}^{g}(f(z))$, for all $g\in G$, $z\in\mathcal{Z}$, where because $G$ acts on $\mathcal{Z}$ we obtain $f_1({}^{g}z)={}^{g}(f_1(z))$, for all $g\in G$, $z\in\mathcal{Z}$, hence $f_1({}^{g^{-1}}z)={}^{g^{-1}}(f_1(z))$, for all $g\in G$, $z\in\mathcal{Z}$. By conjugation with $g$, with obtain $({}^g f_1)(z)=f_1(z)$, for all $g\in G$, $z\in\mathcal{Z}$. This concludes that $f_1\in \mathrm{Aut}_{k}(\mathcal{Z})^G$.

Therefore, we define $\psi:\mathrm{Aut}^{\mathrm{gr}}_{kG}(\mathcal{C})\to  \mathrm{Aut}_{k}(\mathcal{Z})^{G}$, $\psi=(-)_1$. As this construction is functorial, it is clear that $\psi$ is a group homomorphism.

\textit{Step 3: Exactness at $\mathrm{Aut}^{\mathrm{gr}}_{kG}(\mathcal{C})$.} We want to prove that $\mathrm{Im}\,\phi=\mathrm{Ker}\,\psi$. For $\mathrm{Im}\,\phi\subseteq\mathrm{Ker}\,\psi$, consider $\beta\in \mathrm{Hom}_{\mathrm{Grp}}(H,Z)^{G/H}$. Because $Z=\mathrm{U}(\mathcal{Z})$, we clearly have $\beta(1_G)=1_A$, hence $f_{\beta}(z)=z$, for all $z\in \mathcal{Z}=\mathcal{C}_1$, which implies $f_{\beta}\in \mathrm{Ker}\,\psi$.

Conversely, consider $f\in \mathrm{Aut}^{\mathrm{gr}}_{kG}(\mathcal{C})$ such that $f(z)=z$, for all $z\in \mathcal{Z}$. We define $\beta:H\to Z$, $\beta(h)=u_{h}^{-1}f(u_h)$, where $u_h\in\mathcal{C}_h\cap \mathrm{hU}(A)$, for all $h\in H$. Because $f$ preserves grades and invertibles,  $\beta$ is well-defined. We prove that $\beta$ does not depend on the choice of invertible homogeneous elements $u_h\in\mathcal{C}_h\cap \mathrm{hU}(A)$, $h\in H$. Indeed, if $v_h\in\mathcal{C}_h\cap \mathrm{hU}(A)$, there exists $b\in Z=\mathrm{U}(\mathcal{Z})$ such that $v_h=bu_h$. We have
\[\begin{array}{rcccl}
     v_h^{-1}f(v_h) & = & u_h^{-1}b^{-1}f(bu_h)
      & = & u_h^{-1}b^{-1}f(b)f(u_h),
\end{array}\]
but $f(b)=b$, so the claim is proved. Moreover, $\beta$ is group homomorphism. Indeed, for $h_1,h_2\in H$ and $u_{h_1},u_{h_2}\in \mathcal{C}_h\cap \mathrm{hU}(A)$, we have:
\[    \begin{array}{rcccl}
         \beta(h_1h_2) & = & (u_{h_1}u_{h_2})^{-1}f(u_{h_1}u_{h_2})
          & = & u_{h_2}^{-1}u_{h_1}^{-1}f(u_{h_1})f(u_{h_2}),
    \end{array}\]
hence $u_{h_1}^{-1}f(u_{h_1})\in B$, and because $u_{h_2}^{-1}\in \mathcal{C}_h\cap \mathrm{hU}(A)$ we obtain:
\[    \begin{array}{rcccl}
         \beta(h_1h_2) & = & u_{h_1}^{-1}f(u_{h_1})u_{h_2}^{-1}f(u_{h_2})& = & \beta(h_1)\beta(h_2).
    \end{array}\]

Lastly, for $g\in G$, $h\in H$, $u_g\in A_{g}\cap \mathrm{hU}(A)$ and $u_{h}\in \mathcal{C}_h\cap \mathrm{hU}(A)$, we have:
\[    \begin{array}{rcl}
         ({}^{g}\beta)(h) & = & {}^{g}(\beta({}^{g^{-1}}h))\\
          & = & u_g\beta(g^{-1}hg)u_g^{-1}\\
          & = & u_g(u_g^{-1}u_hu_g)^{-1}f(u_g^{-1}u_hu_g)u_g^{-1},\text{ since }u_g^{-1}u_hu_g\in\mathcal{C}_{g^{-1}hg}\cap\mathrm{hU}(\mathcal{C})\\
          & = & u_gu_g^{-1}u_h^{-1}u_gf(u_g^{-1}u_hu_g)u_g^{-1}\\
          & = & u_h^{-1}u_gf(u_g^{-1}u_hu_g)u_g^{-1}\\
          & = & u_h^{-1}\cdot {}^{g}(f(u_g^{-1}u_hu_g))\\
          & = & u_h^{-1}f({}^{g}(u_g^{-1}u_hu_g)),\text{ since }f\in\mathrm{Aut}^{\mathrm{gr}}_{kG}(\mathcal{C})\\
          & = & u_h^{-1}f(u_{g}(u_g^{-1}u_hu_g)u_{g}^{-1})\\
          & = & u_h^{-1}f(u_h)\\
          & = & \beta(h),
    \end{array}\]
therefore $\beta$ is $G$-invariant. Together, these prove that $\beta \in \mathrm{Hom}_{\mathrm{Grp}}(H,Z)^{G/H}$. For $h\in H$ and $c_h\in\mathcal{C}_h$, we have
\[\begin{array}{rcl}
     \phi(\beta)(c_h) & = & f_{\beta}(c_h)\\
     & = & \beta(h) c_h\\
     & = & u_h^{-1} f(u_h) c_h\\
     & = & u_h^{-1} f(u_h) u_h u_h^{-1} c_h\\
     & = & {}^{h^{-1}}f(u_h) u_h^{-1} c_h\\
     & = & f({}^{h^{-1}}u_h) u_h^{-1} c_h,\text{ since }f\in\mathrm{Aut}^{\mathrm{gr}}_{kG}(\mathcal{C})\\
     & = & f(u_h^{-1}u_hu_h) u_h^{-1}t c_h\\
     & = & f(u_h)u_h^{-1} c_h\\
     & = & f(u_h) f(u_h^{-1} c_h),\text{ since }u_h^{-1} c_h\in\mathcal{Z}\\
     & = & f(u_h u_h^{-1} c_h)\\
     & = & f(c_h),\\
\end{array}\]
hence $\phi(\beta)=f$, which shows that $f\in\mathrm{Im}\,\phi$.

The first and last column are given by the 5-term exact sequence \ref{ss:first_column}, and the second column is given by Proposition \ref{prop:gr_pic_action}. For the third column, from \cite[Theorem 37.18]{book:Reiner2003}, we have the  exact sequence
\[\xymatrix@C+=1cm{
1 \ar@{->}[r] & \mathrm{Picent}(B) \ar@{->}[r]^{\subseteq} & \mathrm{Pic}_k(B)\ar@{->}[r]^{\bar{\sigma}} & \mathrm{Aut}_{k}(\mathcal{Z}),
}\]
where $\bar{\sigma}([P])={\sigma}_{P}:\mathcal{Z}\to\mathcal{Z}$ is an automorphism of $k$-algebras, determined by ${\sigma}_{P}(z)p=pz$ for all $p\in P$ and $z\in\mathcal{Z}$.

Note that ${}^{g}\sigma_P=\sigma_{{}^gP}$, for all $g\in G$. Indeed, if $g\in G$, $u_g\in \mathrm{hU}(A)\cap A_g$, $p\in P$ and $z\in\mathcal{Z}$, we have:
\[
\begin{array}{rcl}
     ({}^{g}\sigma_P)(z)(u_g\otimes_B p\otimes_B u^{-1}_g) & = & {}^{g}(\sigma_P({}^{g^{-1}}z))(u_g\otimes_B p\otimes_B u^{-1}_g) \\
     & = & {u}_{g}\sigma_P(u_g^{-1}zu_g)u_g^{-1}(u_g\otimes_B p\otimes_B u^{-1}_g) \\
     & = & {u}_{g}\sigma_P(u_g^{-1}zu_g)u_g^{-1}u_g\otimes_B p\otimes_B u^{-1}_g \\
     & = & {u}_{g}\sigma_P(u_g^{-1}zu_g)\otimes_B p\otimes_B u^{-1}_g \\
     & = & {u}_{g}\otimes_B \sigma_P(u_g^{-1}zu_g)p\otimes_B u^{-1}_g,\text{ since }\sigma_P(u_g^{-1}zu_g)\in\mathcal{Z}  \\
     & = & {u}_{g}\otimes_B p (u_g^{-1}zu_g)\otimes_B u^{-1}_g\\
     & = & {u}_{g}\otimes_B p \otimes_B (u_g^{-1}zu_g)u^{-1}_g\\
     & = & {u}_{g}\otimes_B p \otimes_B u_g^{-1}z\\
     & = & ({u}_{g}\otimes_B p \otimes_B u_g^{-1})z,
\end{array}
\]
which proves the claim. This implies that if $[P]\in\mathrm{Pic}_k(B)^G$, that is, $[P]={}^{g}[P]$, for all $g\in G$; then ${}^{g}\sigma_P=\sigma_P$, for all $g\in G$, so $\sigma_P\in \mathrm{Aut}_{k}(\mathcal{Z})^G$. Therefore, we define $\sigma:\mathrm{Pic}_k(B)^G\to \mathrm{Aut}_{k}(\mathcal{Z})^G$ to be the restriction of $\bar{\sigma}$ to $\mathrm{Pic}_k(B)^G$.  We clearly have that $\mathrm{Picent}(B)^{G/H} \subseteq \mathrm{Pic}_k(B)^G$, and it is obvious that the new sequence is exact at $\mathrm{Pic}_k(B)^G$.

By the definitions of the maps $\Phi$, $\Psi$, $\psi$ and $\Theta$, it is straightforward to see that the squares formed by the first two rows, and also the bottom right square are commutative. By the definition of $\Phi_k$, to the $1$-cocycle $\beta:G\to Z$ corresponds the $G$-graded $(A,A)$-bimodule $\tilde P$ with $1$-component $\tilde P_1=B^\beta=B$, where $bc_h=\beta(h)c_hb$ for all $b\in B$, $h\in H$ and $c_h\in \mathcal{C}_h$. This means that $\tilde\sigma_{\tilde P}(c_h)=  \beta(h)c_h=\phi(\beta)(c_h)$, hence the bottom left square is also commutative.


2)  Recall that the transgression map $\mathrm{Tg}:\mathrm{H}^1(H,Z)^{G/H}\to\mathrm{H}^2(G/H,Z)$ can be defined as follows.
For any $x\in G/H$, choose $s_x\in G$ such that $\mathrm{can}(s_x)=x$, where $\mathrm{can}:G\to G/H$ is the canonical epimorphism. Let $s:G/H\times G/H\to H$ be defined by $s_xs_y=s(xy)s_{xy}$, for all $x,y\in G/H$.
Let $\beta\in \mathrm{H}^1(H,Z)^{G/H}$, then $\beta\circ s:G/H\times G/H\to Z$ is a 2-cocycle and $\mathrm{Tg}([\beta])=[\beta\circ s]\in \mathrm{H}^2(G/H,Z)$.

We define the group homomorphism $\theta:\mathrm{H}^1(H,Z)^{G/H}\to \mathrm{Pic}^{\mathrm{gr}}_{\mathcal{Z}}(A_H)^{G/H}$. Let $P=B$; as in the definition of $\Phi$, the multiplications $\mu_g:A_{g^{-1}}\otimes_B B\otimes_B A_g \to B$, for $g\in H$, give a splitting for $\mathrm{hU}(E_H)$. Let $\beta\in \mathrm{H}^1(H,Z)^{G/H}$, and denote by $P^\beta$ the $G/H$-invariant $\Delta_H$-module extending $P$ corresponding to $\beta$. Let $\hat P^\beta = A_H\otimes_BP^\beta$, and define $\theta([\beta])=[\hat P^\beta]\in \mathrm{Pic}^{\mathrm{gr}}_{\mathcal{Z}}(A_H)^{G/H}$.

As in the proof of Theorem \ref{th:main_exact_sequnce}, consider the $G/H$-graded endomorphism algebra $\bar E^\beta=\mathrm{End}_\Delta(\Delta\otimes_{\Delta_H}P^\beta)^{\mathrm{op}}$ with 1-component $\mathcal{Z}$. This gives us the element $\Theta([A_H\otimes_B P^\beta])\in\mathrm{H}^2(G/H,Z)$, and it is easy to see that
$\Theta([A_H\otimes_B P^\beta])\in\mathrm{H}^2(G/H,Z)=\mathrm{Tg}([\beta]).$
\end{proof}

\begin{theorem} \label{th:butterfly}     Let $\hat{G}$ be another finite group and let $\hat{A}$ be a $\hat{G}$-graded crossed product with $1$-component $B$. Assume that the maps     $\varepsilon:\mathrm{hU}(A)\to \mathrm{Aut}_k(B)$     and      $\hat{\varepsilon}:\mathrm{hU}(\hat{A})\to \mathrm{Aut}_k(B)$  induced by conjugation satisfy $\mathrm{Im}\, \varepsilon=\mathrm{Im} \, \hat{\varepsilon}$. Denote $\hat{H}=\hat{G}[B]$. Then $G/H\simeq \hat{G}/\hat{H}$, and we have an isomorphism $\alpha:\mathrm{Picent}^{\mathrm{gr}}(A)\to\mathrm{Picent}^{\mathrm{gr}}(\hat{A})$ such that the following diagram is commutative:
        \[\xymatrix@C+=1cm{
1 \ar@{->}[r] & \mathrm{H}^1(G/H, Z) \ar@{->}[r]^{\Phi} \ar@{->}[d]^{\simeq}& \mathrm{Picent}^{\mathrm{gr}}(A)\ar@{->}[r]^{\Psi} \ar@{..>}[d]^{\alpha} & \mathrm{Picent}(B)^{G/H}\ar@{->}[r]^{\Theta} \ar@{=}[d] & \mathrm{H}^2(G/H, Z) \ar@{->}[d]^{\simeq}\\
1 \ar@{->}[r] & \mathrm{H}^1(\hat{G}/\hat{H}, Z) \ar@{->}[r]^{\hat{\Phi}} & \mathrm{Picent}^{\mathrm{gr}}(\hat{A})\ar@{->}[r]^{\hat{\Psi}}  & \mathrm{Picent}(B)^{\hat{G}/\hat{H}}\ar@{->}[r]^{\hat{\Theta}}  & \mathrm{H}^2(\hat{G}/\hat{H}, Z).
}\]

\begin{proof}  By the commutative diagram of \ref{ss:cvasihuge_diagram} and by our hypothesis that $\mathrm{Im}\,\varepsilon=\mathrm{Im}\,\hat{\varepsilon}$, we obtain $\mathrm{Im}\,\bar{\varepsilon}=\mathrm{Im}\,\bar{\hat{\varepsilon}}$, and hence the  group isomorphism $G/H\simeq \hat{G}/\hat{H}.$ This induces the isomorphisms $\mathrm{H}^i(G/H, Z)\simeq \mathrm{H}^i(\hat{G}/\hat{H}, Z)$, for $i=1,2$, and we will identify these groups.

Let $[\tilde{P}]\in \mathrm{Picent}^{\mathrm{gr}}(A)$, and let $P:=\tilde{P}_1$, hence $[P]=\Psi([\tilde{P}])\in\mathrm{Picent}(B)$, and $\tilde{P}\simeq A\otimes_B P$. Consider the diagram
\[    \xymatrix@C+=2cm@R+0.6cm{
\hat{A} \ar@{.}@/^-2.0pc/[rrr]_{\sim}^{\tilde{\hat{P}}=\hat{A}\otimes_B P} & A \ar@{-}[r]^{\tilde{P}=A\otimes_B P}_{\sim} & A & \hat{A}\\
	\hat{C}=\hat{A}_{\hat{H}} \ar@{-}@/^-2.0pc/[rrr]_{\sim}^{\hat{C}\otimes_B P} \ar@{-}[u] & C=A_H \ar@{-}[u] \ar@{-}[r]_{\sim}^{C\otimes_B P}& C=A_H \ar@{-}[u] & \hat{C}=\hat{A}_{\hat{H}} \ar@{-}[u]\\
	 &B \ar@{-}[u] \ar@{-}[r]^{P=\tilde{P}_1}_{\sim} \ar@{-}[ul]& B. \ar@{-}[u] \ar@{-}[ur]&
    }\]
By Proposition \ref{prop:eq}, we obtain that $C\otimes_B P\in\mathrm{Picent}^{\mathrm{gr}}(C)$, and that $\hat{C}\otimes_B P\in\mathrm{Picent}^{\mathrm{gr}}(\hat{C})$.

Denote $\hat\Delta=\Delta(\hat{A}\otimes \hat{A}^{\mathrm{op}})$, hence $\hat\Delta_{\hat H}=\Delta(\hat{C}\otimes \hat{C}^{\mathrm{op}})$. We already know by Proposition \ref{prop:eq} that $P$ is a $\hat\Delta_{\hat H}$-module, and we need to define a $\Delta(\hat{A}\otimes \hat{A}^{\mathrm{op}})$-module structure on $P$.    Let $\mathcal{T}$ be a complete set of representatives for the cosets of $H$ in $G$. The isomorphism $G/H\simeq \hat{G}/\hat{H}$ gives us a complete set of representatives for the cosets of $\hat{H}$ in $\hat{G}$ such that $t\in\mathcal{T}$ corresponds to $\hat{t}\in\hat{\mathcal{T}}$. Moreover, by \ref{ss:cvasihuge_diagram}, \[G/H\simeq \mathrm{hU}(A)/\mathrm{U}(B)\mathrm{C}_{\mathrm{hU}(A)}(B)\] and
\[\hat{G}/\hat{H}\simeq \mathrm{hU}(\hat{A})/\mathrm{U}(B)\mathrm{C}_{\mathrm{hU}(\hat{A})}(B),\]
so we can find invertible homogeneous elements $u_t\in \mathrm{hU}(A)\cap A_t$ and $\hat{u}_{\hat{t}}\in\mathrm{hU}(\hat{A})\cap \hat{A}_{\hat{t}}$ such that $\varepsilon(u_t)=\hat{\varepsilon}(\hat{u}_{\hat{t}})$.

Since $[P]\in\mathrm{Picent}(B)^{G/H}$, we have that for all $t\in\mathcal{T}$, $u_t\otimes P\otimes u^{-1}_t\simeq P$     as $\Delta_H$-modules. This means that there is a bijection $f_t:P\to P$ such that this isomorphism is given by $u_t\otimes p\otimes u^{-1}_t\mapsto f_t(p)$. Since $\varepsilon(u_t)=\hat{\varepsilon}(\hat{u}_{\hat{t}})$, we obtain that the map
\[\hat{u}_{\hat{t}}\otimes P\otimes \hat{u}^{-1}_{\hat{t}}\to P, \qquad \hat{u}_{\hat{t}}\otimes p\otimes \hat{u}^{-1}_{\hat{t}}\mapsto f_t(p)\]
is also an isomorphism of $\hat\Delta_{\hat H}$-modules. It follows that $[P]\in\mathrm{Picent}(B)^{\hat{G}/\hat{H}}$, and moreover, $\Theta([P])=\hat\Theta([P])$. This argument also shows that the third square is commutative.

In our situation, $\Theta([P])=1$, hence $\hat\Theta([P])=1$. It follows that we may assume that the elements $u_t$ give a splitting for $\mathrm{hU}(\bar{E})$, where $\bar{E}=\mathrm{End}_{\Delta}(\Delta\otimes_{\Delta_H}P)^{\mathrm{op}}$, and the elements $\hat{u}_{\hat{t}}$ give a splitting for $\mathrm{hU}(\bar{\hat{E}})$, where $\bar{\hat{E}}=\mathrm{End}_{\hat{\Delta}}(\hat{\Delta}\otimes_{\hat{\Delta}_{\hat{H}}}P)^{\mathrm{op}}$. By Dade's theorem \cite[Theorem 2.8]{art:Dade1981} the latter induces a $\hat{\Delta}$-module structure on $P$, hence a structure of a $\hat{G}$-graded $(\hat{A},\hat{A})$-bimodule over $\hat{\mathcal{C}}$ on $\tilde{\hat{P}}:=\hat{A}\otimes_B P$. Then, $[\tilde{\hat{P}}]\in\mathrm{Picent}^{\mathrm{gr}}(\hat{A})$, and we define $\alpha([\tilde{P}])=[\tilde{\hat{P}}]$.

It is clear that $\alpha$ is a well-defined group isomorphism, and that the second square is commutative. By the definition of $\Phi$, it also follows that the first square is commutative as well.
 \end{proof}
\end{theorem}

\subsection*{Acknowledgement} This research is supported by a grant of the Ministry of Research, Innovation and Digitization, CNCS/CCCDI--UEFISCDI, project number PN-III-P4-ID-PCE-2020-0454, within PNCDI III.


\end{document}